\documentclass[12pt,oneside,reqno]{amsart}

\usepackage{amssymb}
\usepackage[active]{srcltx}
\usepackage{a4wide}
\usepackage{amsmath,amsthm}
\usepackage{amssymb}
\usepackage{amstext}
\everymath{\displaystyle}
\usepackage{cite}
\usepackage{epsfig}
\usepackage{enumerate}
\usepackage{color}
 \usepackage{setspace} 
\usepackage{amssymb}
\newtheorem{theorem}{Theorem}[section]

\theoremstyle{definition}
\newtheorem{definition}[theorem]{Definition}

\theoremstyle{remark}

\usepackage{hyperref}
\hypersetup{colorlinks,allcolors=black}

\newcommand{\be}{\begin{equation}}
\newcommand{\ee}{\end{equation}}

\begin{document}

\title[Some analytical properties of  fractal functions in Lebesgue spaces]{Some analytical properties of multivariate fractal functions in Lebesgue spaces}

\author{Kiran Rani}
\address{Department of Mathematics, 
Punjab Engineering college
(Deemed to be University), 
Sector 12 Chandigarh 160012, India}
\email{kirankaushik2310@gmail.com}

\author{Rattan Lal}
\address{Department of Mathematics, 
Punjab Engineering college
(Deemed to be University), 
Sector 12 Chandigarh 160012, India}
\email{rattanlal@pec.edu.in}



\begin{abstract}
In this article, we focus on the construction of multivariate fractal functions in Lebesgue spaces along with some properties of associated fractal operator. First, we give a detailed construction of the fractal functions belonging to Lebesgue spaces. 
Then, we give analytical properties of the defined fractal operator in Lebesgue spaces. We end this article by showing the existence of Schauder basis of the associated fractal functions for the space $\mathcal{L}^q(I^n, \mu_p)$.
\end{abstract}

\maketitle

\section{Introduction and Preliminaries} 
\noindent
Focusing on the main goal of constructing an attractor, Barnsley \cite{MF1, MF2} works investigates the generation of FIFS using IFS. The creation of surfaces in $R^3$ that operate as attractor for IFS is addressed  by Massopust \cite{Mas4}. The graphs of continuous functions from the oriented standard simplex $\sigma^2 \subset R^2$ into R are expressed using these surfaces.

Geronimo and Hardin \cite{GH} provide the provision of an algorithm for constructing FIS over polygonal regions with arbitrary interpolation points. Also, a class of invariant measures supported on these FIS is introduced. In another work \cite{Zhao}, Zhao presents an approach to construct the FIS using triangulation. Moreover, theoretically \cite{Zhao} proves that this constructions attractor are continuous FISs. Dalla \cite{Dalla} discussed a method for construction of bivariate FIFs by investigating whether the interpolation points on the edge are collinear. Later, he discussed the general case.
The fractal interpolation approach on a rectangular domain is revistied in \cite{Feng}. In contrast to conventional methods, it does not rely on similar/affine mappings in the Fractal Interpolation Function (FIF) or equal-spaced interpolation points.

 For each set of data, Malysz \cite{Mz} introduce a novel construction of the continuous bivariate fractal interpolation surface. Metzler and Yun
\cite{MYUN} provide a structure that incorporates a vertical contraction factor function utilizing a more generalized Iterated Function System (IFS). Ruan and Xu \cite{Ruan} provide a unique category of FISs called bilinear FISs. 
 Two bivariate functions formed on a rectangle are the source of the fractal surfaces shown in \cite{NMA}. There are no particular requirements that the data must meet.
 As a generalization,  Bouboulis et al. \cite{D2} present recurrent BFISs to increase the flexibility of image compression or natural-shape generation. In \cite{D1}, Bouboulis and Dalla present a method for creating FIFs that generates the fractal interpolant of a few predetermined points of $\mathbb{R}^n$ by utilizing ordinary interpolants of points of $\mathbb{R}^{n-1}$.
 
  Work in \cite{M21} proposes to modify conventional polynomials to create fractal interpolants. The study also suggests a technique for interpolating real data by building a fractal interpolation function from a classical approximant. In \cite{M2} M. A. Navascu\'es deals with Fractal Approximation.  Further \cite{SV, SP} explored more fractal approximation aspects.
Through the use of oscillation spaces and Hölder space, \cite{Mnew2} create non-stationary fractal surfaces across a rectangular domain and investigate the dimension of non-stationary FISs.
 Starting with the  construction of bivariate FIFs and FISs, \cite{SS4} study partial integrals and partial derivatives of these functions. Authors in \cite{VV3} re-examines the creation of bivariate fractal interpolation functions, aiming to develop a family of parameterized fractal functions that align with a given bivariate continuous function over a rectangular area in $R^2$.
 The results in \cite{NVV2} allows us to create a family of fractal functions connected to a continuous vector-valued function that is specified and described on the Sierpi\'{n}ski gasket. Later \cite{PV11} describes FIFs on the product of Sierpi\'{n}ski gasket.
  
 The estimations for the dimensions of the invariant measures related to fractal interpolation surfaces are discussed in \cite{ES}. Further, the fractal surface's projection, limitation, and associated measurements are explored.
 
 In view of a few complete spaces, \cite{KPV} presented a unique class of FIFs in multivariate known as fractal functions. Also, the conditions on the defining parameters for a class of fractal functions defined on a hyperrectangle in the Euclidean space $R^n$, such that the fractal functions are elements of some standard function spaces, such as Lebesgue spaces, the Sobolev spaces, and the Hölder spaces are discussed. Further in \cite{ES1}, Agrawal and Verma studied fractal surfaces on Sobolev and  H\"{o}lder spaces.
  For a given set of data points, \cite{VMT} create multivariate fractal interpolation functions and investigate the existence of the $\alpha$-fractal function that corresponds to the multivariate continuous function. This paper investigates the $\alpha$-fractal function's restriction on the coordinate axis. Additionally, the multivariate $\alpha -$fractal function's graph's box and Hausdorff dimensions, as well as its restriction, are examined. In the space of continuous functions, \cite{GVS} demonstrate the existence of a novel class of a-fractal functions devoid of endpoint conditions. They also include the existence of the same class in many spaces, including the oscillation space, the convex Lipschitz space, and the Hölder space. Next, \cite{RSP} investigates fractal surfaces and the associated fractal operator in Lebesgue spaces with respect to fractal measures. This paper deals with bivariate fractal functions. As we see in the literature survey it is still open to deal with multivariate fractal interpolation functions in Lebesgue spaces, so this work \cite{RSP} becomes the basic block to our current paper. The organisation of this paper is as follows: Firstly, we start with some definitions. Then in section 2, we follow the construction given by Ruan and Xu \cite{Ruan} to construct multivariate fractal interpolation functions. Next in Section 3, our focus is on Lebesgue spaces along with a key theorem. Then we discuss the associated fractal operator and its properties. Next, we explore fractal interpolation by using a specific germ function and a base function. Lastly, we discuss about the Schauder basis.  
 \newline Now, let us become familiar with some terms that we shall be using throughout the paper.
   
%\begin{definition} \label{Iterated}
 For the definition and examples of IFS, self-similar IFS and attractor,  refer to \cite{Fal}.
%\end{definition}
\begin{definition}\cite{KRZ}
   Let $X$ be a normed space which contains a sequence $(x_m)$. If for every $y \in X$ there is a unique sequence of scalars $(\beta_m)$such that

\begin{equation*}
\| y - (\beta_1 x_1 + \cdots + \beta_m x_m) \| \rightarrow 0 \quad \text{(as } m \rightarrow \infty \text{),}
\end{equation*}

then $(x_n)$ becomes a Schauder basis (or basis) for $X$ .
\end{definition}
 %\begin{definition}\emph{\cite{Fal}}\label{OSC}
%An IFS $\big\{X; \upsilon_1, \ldots, \upsilon_m\big\}$ is said to satisfy \textit{open set condition} (OSC) if there exists a non-empty open set $V\subset \mathbb{R}$ such that 
%\[\bigcup_{i=1}^{m}\upsilon_i(V)\subset V ~\text{ such that }~ \upsilon_i(V)\cap \upsilon_j(V)=\emptyset ~\text{ for }~ i\neq j.\]
%Moreover, if $\mathcal{A}$ is the attractor of this IFS such that $V\cap \mathcal{A}\neq \emptyset$, then the IFS is said to satisfy the \textit{strong open set % \end{definition}

%\begin{definition}\emph{\cite{MF2}}
%Let $(X,d)$ be a compact metric space and $\nu$ be a Borel measure on $X$; if $\nu(X)=1$, then $\nu$ is said to be Borel probability measure.
%\end{definition}

\begin{theorem}\emph{\cite{MF2}}
Let $(X,d)$ be a complete metric space and $\{X; \upsilon_1,\ldots \upsilon_m\}$ be the IFS corresponding to given metric space. Let $\boldsymbol{r}=(r_1,\ldots, r_m)$ be a vector of probability. Then, there exists a unique Borel probability measure $\mu_r$ such that 
\[\mu_r=\sum_{i=1}^{m}r_i \mu_r \circ \upsilon_{i}^{-1}.\]
Moreover, the support of $\mu_r$ is the attractor of the IFS.
\end{theorem}
%Define $ \Theta_{\upsilon} = \bigcup_{i\neq j} (\upsilon_i(\mathcal{A})\cap \upsilon_j(\mathcal{A})) .$ By \cite{Moran}, we have $\mu_r(\Theta_{\upsilon})=0$ for self-similar measures %and under open set condition.
%%%%%%%%%%%%%%%%%%%%%%%%%%%%%%%%%%%%%%%%%%%%%%%%%%%%%%%%%%%%%%%%%%%%%%%%%%%%%%%%%%%
%%%%%%%%%%%%%%%%%%%%%%%%%%%%%%%%%%%%%%%%%%%%%%%%%%%%%%%%%%%%%%%%%%%%%%%%%%%%%%%%%%%

\section{Preliminaries on Multivariate Fractal Interpolation function }   
For more specification of the topic, one can visit \cite{Ruan}. 
Let $n\geq2$ be a naural number.
Consider $\bigl\{(y_{1,j_1}, y_{2,j_2},\cdots,y_{n,j_n},x_{j_1 j_2 \cdots j_n})\in \mathbb{R}^{n+1}$:$(j_1,j_2,\cdots,j_n)\in\prod_{i=1}^{n} \Sigma_{N_i}\bigr\}$  the interpolation data set such that $y_{k,0}<y_{k,1}<\cdots<y_{k,N_k}$ for each k, $k=1,2,\cdots,n$. For any positive integer N,  $\Sigma_N=\{1,2,\dots,N\}, ~ \Sigma_{N,0}=\{0,1,\dots,N\}, ~ \partial\Sigma_{N,0}=\{0,N\} \text{ and }  \text{int}\Sigma_{N,0}=\{1,2,\dots,N-1\}$.  
Set $I_{k,i_k}=[y_{k,j_{k-1}} , y_{k,j_k}]$ for $j_k=1,2,\cdots,N_k$ and $I_k=\bigcup_{j_k=1}^{N_k} I_{k,j_k}$ for $k\in \Sigma_n$ . Further, for $j_k\in \Sigma_{N_k}$, define $v_{k,j_k}:I_k\to I_{k,j_k}$ with
\begin{equation} \label{D1}
    \begin{split}
     &v_{k,j_k}(y_{k,0})=y_{k,j_k-1},~v_{k,j_k}(y_{k,N_k})=y_{k,j_k}, \ \text{if}\ j_k \ \text{is odd}\\
     &v_{k,j_k}(y_{k,0})=y_{k,j_k},~v_{k,j_k}(y_{k,N_k})=y_{k,j_k-1}, \ \text{if} \ j_k \ \text{is even},\\
    &\lvert v_{k,j_k}(z)-v_{k,j_k}(z_*) \rvert\leq \gamma_{k,j_k}\lvert z-z_* \rvert, \ \text{for}\ z,z_*\in I_k,\\
 \end{split}
\end{equation}  
where $0<\gamma_{k,j_k}<1$ is a contraction factor.
It is easy to verify that
\begin{equation*}
    \begin{split}
    &v_{k,j_k}^{-1}(y_{k,j_k})=v_{k,j_k+1}^{-1}(y_{k,j_k}), \text{ for }  j_k\in \text{int}\Sigma_{N_k,0},k\in \Sigma_n \\
    \end{split}
\end{equation*}
For simplification, we have to define $\tau:\mathbb{Z}\times\{0,N_1,N_2,\cdots,N_n\} \to \mathbb{Z}$ as follows:
\begin{equation*}
    \begin{split}
     \tau(j,0)= \left\{\begin{array}{rcl}j-1, & \mbox{if} \  j \ \mbox{is odd,}\\j, & \mbox{if} \ j \ \mbox{is even,}\end{array}\right. \  \text{and} \\
     \tau(j,N_1)=\tau(j,N_2)=\cdots=\tau(j,N_n)= \left\{\begin{array}{rcl}j, & \mbox{if} \  j \ \mbox{is odd,}\\j-1, & \mbox{if} \ j \ \mbox{is even.}\end{array}\right.
    \end{split}
\end{equation*}
It is easy to see that $v_{k,j_k}(y_{k,i_k})=y_{k,{\tau(j_k,i_k)}},\forall\ j_k\in \Sigma_{N_k}, i_k\in\partial\Sigma_{N_k,0}, k\in \Sigma_n$ . Let $M=I^{n}\times\mathbb{R}$. Moreover, for each $(j_1,j_2,\cdots,j_n)\in\prod_{i=1}^{n} \Sigma_{N_i}$, define a continuous function $H_{j_1,j_2,\cdots,j_n}:M\to\mathbb{R}$ fulfilling
\begin{equation*}
    \begin{split}
       H_{j_1 j_2 \cdots j_n}(y_{1,k_1},y_{2,k_2},\cdots,y_{n,k_n},x_{(1,k_1)\cdots(n,k_n)})=x_{(1,\tau(j_1,k_1))\cdots(n,\tau(j_n,k_n))}, \\  \text{for} \ (k_1,k_2,\cdots,k_n)\in \prod_{i=1}
^n\partial\Sigma_{N_i,0}, \text{ and}\\
        \lvert H_{j_1 j_2 \cdots j_n}(y_1,y_2,\cdots,y_n,z)-H_{j_1 j_2 \cdots j_n}(y_1,y_2,\cdots,y_n,z_*) \rvert \leq\delta_{j_1 j_2\cdots j_n} \lvert z-z_*\rvert,\\  \text{for} \ (y_1,y_2,\cdots,y_n) \in I^{n}; \ z, z_*\in \mathbb{R},
    \end{split}
\end{equation*}
where $0<\delta_{j_1 j_2\cdots j_n}<1$ is a contraction factor. \\
Define 
\begin{equation*}
\begin{split}
\mathcal{C}_*(I^{n}):=\bigl\{g\in\mathcal{C}(I^{n})&:g(y_{1,j_1},y_{2,j_2},\cdots,y_{n,j_n})=x_{(1,j_1)\cdots(n,j_n)},\\& \forall~(j_1,j_2,\cdots,j_n)\in\prod_{i=1}^{n} \Sigma_{N_i,0}\bigr\}.
\end{split}
\end{equation*}
%where $\mathcal{C}(I\times J)$ is the set of all real-valued continuous functions defined on $I\times J.$ It is notable that $\mathcal{C}_*(I\times J)$ endowed with the sup norm is a Banach space.
Let $T$ be a Read-Bajraktarevi$\acute{c}$ (RB) operator on $\mathcal{C}_*(I^{n})$ as follows
$$T(g)(y_1,y_2,\cdots,y_n):=H_{j_1 j_2 \cdots j_n}(v_{1,j_1}^{-1}(y_1),\cdots,v_{n,j_n}^{-1}(y_n),g(v_{1,j_1}^{-1}(y_1),\cdots,v_{n,j_n}^{-1}(y_n))),$$ 
for $(y_1,y_2,\cdots,y_n)\in \prod_{k=1}^{n} I_{k,j_k},~(j_1,j_2,\cdots,j_n)\in\prod_{i=1}^{n} \Sigma_{N_i,0}$.\\  With factor $\max_{j_1 j_2\cdots j_n}(\delta_{j_1 j_2\cdots j_n})$,  $T$ becomes a contraction. So, by Banach fixed point theorem,  $T$ have a unique fixed point, say, $h$ which interpolates the given data set and meets the upcoming equation 
\begin{equation*}
\begin{split}
&h(y_1,y_2,\cdots,y_n)=H_{j_1 j_2\cdots j_n}(v_{1,j_1}^{-1}(y_1),\cdots,v_{n,j_n}^{-1}(y_n),h(v_{1,j_1}^{-1}(y_1),\cdots,v_{n,j_n}^{-1}(y_n))), \\ & \text{for} ~ (y_1,y_2,\cdots,y_n)\in \prod_{k=1}^{n} I_{k,j_k},~(j_1,j_2,\cdots,j_n)\in\prod_{i=1}^{n} \Sigma_{N_i}. \\& \text{Eventually, for\ each}~(j_1,j_2,\cdots,j_n)\in\prod_{i=1}^{n} \Sigma_{N_i}, \\& \text{Define} \ S_{j_1 j_2\cdots j_n}:M\to \prod_{k=1}^nI_{k,j_k}\times\mathbb{R}\ \text{as} 
\\& S_{j_1 j_2\cdots j_n}(y_1,y_2,\cdots,y_n,x)=\bigl(v_{1,j_1}(y_1),\cdots,v_{n,jn}(y_n),H_{j_1 j_2\cdots j_n}(y_1,y_2,\cdots,y_n,x)\bigr),\\& \text{ for }(y_1,y_2,\cdots,y_n,x)\in M.
\end{split}
\end{equation*}
Therefore, $\{M;~S_{j_1 j_2\cdots j_n}:(j_1,j_2,\cdots,j_n)\in\prod_{i=1}^{n} \Sigma_{N_i}\}$ is an IFS. From \cite[Theorem 3.1]{Ruan}, it can be inferred that the attractor of the IFS is $G(h),$ that is,  
$$G(h)=\bigcup_{j_1 j_2 \cdots j_n}S_{j_1 j_2 \cdots j_n}(G(h)),$$ 
where $G(h)=\{(y_1,y_2,\cdots,y_n,h(y_1,y_2,\cdots,y_n)):(y_1,y_2,\cdots,y_n)\in I^{n}\}.$ The function $h$ and $G(h)$ are referred as FIF and FIS in reference to the above IFS, respectively. The functions $v_{k,j_k} \text{ for } (j_1,j_2,\cdots,j_n)\in\prod_{i=1}^{n} \Sigma_{N_i},$ stated above can be selected as 
\begin{equation}\label{D5}
    \begin{aligned}
    & v_{k,j_k}(y)=\alpha_{k,j_k}(y)+ \beta_{k,j_k}, \text{ for } y\in I^{n},\\     
    \end{aligned}
     \end{equation}
here, we can determine constants $\alpha_{k,j_k}, \beta_{k,j_k}$ using \eqref{D1}  as 
\begin{equation*}
    \begin{aligned}
    \left\{\begin{array}{rcl}\alpha_{k,j_k}=\frac{y_{k,j_k}-y_{k,j_{k-1}}}{y_{k,N_k}-y_{k,0}}, & \beta_{k,j_k}=\frac{y_{k,j_k} y_k -y_{k,j_{k-1}}  y_{k,0}}{y_{k,N_k}-y_{k,0}}, & \mbox{if} \ j_k \ \mbox{is odd,} \\ \alpha_{k,j_k}=\frac{y_{k,j_{k-1}}-y_{k,j_k}}{y_{k,N_k}-y_{k,0}}, &  \beta_{k,j_k}=\frac{y_{k,j_{k-1}} y_{k,N_k} -y_{k,j_k}  y_{k,0}}{y_{k,N_k}-y_{k,0}}, & \mbox{if} \ j_k \ \mbox{is even.}\end{array}\right. \ 
    \end{aligned}
\end{equation*}
\begin{center}
    \section{Main Results}
\end{center}
Here, we explore the Lebesgue spaces, with a focus on invariant measures that are related to Iterated Function Systems (IFS). 
\subsection{ Lebesgue spaces with respect to invariant measures}\label{Lpspace}

For simplification, let us denote $\mathcal{I}^{n}=I_1\times I_2 \times \cdots \times I_n.$
Let $\boldsymbol{p}=(p_{11\cdots1},\ldots, p_{N_1 N_2\cdots N_n})$ be a vector of probabilities  and $\mu_p$ be the invariant measure associated with this probability vector generated by the Iterated Function System (IFS) $\{\mathcal I^{n}; v_{k,j_k}:(j_1,\cdots,j_n) \in \prod_{i=1}^{n} \Sigma_{N_i} ,k\in \Sigma_n  \}$ such that
\[\mu_{p}=\sum_{(j_1,\cdots,j_n) \in \prod_{i=1}^{n} \Sigma_{N_i}}p_{j_1 j_2\cdots j_n}\mu_{p}\circ (v_{1,j_1}^{-1},\cdots,v_{n,j_n}^{-1}).\]
To clear the above notation, for $B \subseteq \mathcal{I}^n $ , \[\mu_{p}(B)=\sum_{(j_1,\cdots,j_n) \in \prod_{i=1}^{n} \Sigma_{N_i}}p_{j_1 j_2\cdots j_n}\mu_{p} (\Theta_{j_i j_2\cdots j_n}^{-1}(B)),\]
where $\Theta_{j_i j_2\cdots j_n}^{-1} :\mathcal{I}^{n}\rightarrow \mathcal{I}^{n} $ is defined as \\ $ \Theta_{j_i j_2\cdots j_n} (y_1,y_2,\cdots,y_n)=(v_{1,j_1}(y_1),\cdots,v_{n,j_n}(y_n)) $ \\
For $1\leq q < \infty$, let $L^{q}(\mathcal I^{n},\mu_{p})$  denote the space of all $q$-integrable functions on $\mathcal I^{n}$ associated with $\mu_{p}$, where
\[\mathcal{L}^{q}(\mathcal I^{n},\mu_{p})=\left\{f: \mathcal I^{n} \rightarrow \mathbb{R}: \int_{\mathcal I^{n}}\lvert f(y_1,y_2,\cdots ,y_n)\rvert ^{q}d\mu_{p}(y_1,y_2,\cdots ,y_n)< \infty\right\}.\]
 The normed space $\left(L^{q}(\mathcal I^{n},\mu_{p}), \lVert\cdot \rVert_q\right)$ is a complete normed by $\lVert \cdot \rVert_q$, where 
\[\lVert f \rVert_q=\left(\int_{\mathcal{I}^n}\lvert f(y_1,y_2,\cdots ,y_n)\rvert^q d\mu_{p}(y_1,y_2,\cdots ,y_n)\right)^{1/q}.\]
In relation to the measure $\mu_p$, for $q=\infty$, we establish the essential sup norm $\lVert \cdot \rVert_{\infty}$  such that $\lVert g \rVert_{\infty}=ess\sup \lvert g \rvert$ and in context to this, the space $L^{\infty}(\mathcal I^{n}, \mu_p)=\{g:\mathcal I^{n}  \rightarrow \mathbb{R}:~ \lVert g \rVert_{\infty}< \infty\}$ becomes complete.
\begin{theorem}\label{FIFinLp}
If $\alpha \in L^{\infty}(\mathcal I^{n},\mu_p)\text{ and } q_{j_1 j_2 \cdots j_n}\in L^{q}(\mathcal I^{n},\mu_{p})$ for all $(j_1,\cdots,j_n) \in \prod_{i=1}^{n} \Sigma_{N_i}$, then the associated FIF $h^*\in L^{q}(\mathcal I^{n},\mu_{p})$, provided $\lVert \alpha \rVert_{\infty}< 1$.
\end{theorem}

\begin{proof}
Consider a subset of $L^{q}(\mathcal I^{n},\mu_{p})$ namely\\ $L^{q}_{0}(\mathcal I^{n},\mu_{p})=\left\{g\in L^{q}(\mathcal I^{n},\mu_{p}):~ g(x_1,x_2,\cdots,x_n)=(y_1,y_2,\cdots,y_n) \forall (x_1,x_2,\cdots,x_n) \in \partial \mathcal{I}^{n} \right\}$. One can easily prove  that $L^{q}_{0}(\mathcal I^{n},\mu_{p})$ is a closed subset of the given space $L^{q}(\mathcal I^{n},\mu_{p})$. We know that a closed subspace of a complete normed space is complete. Hence $L^{q}_{0}$ is a complete subspace with respect to the induced metric by $\lVert \cdot \rVert_q$. Define a Read-Bajraktarevi$\acute{c}$ (RB) operator $T:L^{q}_{0}(\mathcal I^{n},\mu_{p})\rightarrow L^{q}_{0}(\mathcal I^{n},\mu_{p})$ such that 
\begin{equation}
    \begin{split}
    (Tg)(y_1,\cdots ,y_n)=& \alpha(y_1,\cdots ,y_n)(g(v_{1,j_1}^{-1}(y_1),\cdots,v_{n,j_n}^{-1}(y_n) \\&
    -q_{j_1 j_2 \cdots j_n}(v_{1,j_1}^{-1}(y_1),\cdots,v_{n,j_n}^{-1}(y_n))),
\end{split}
\end{equation}

for $(y_1,\cdots ,y_n)\in \prod_{i=1}^{n} I_{k,j_k}$ and $(j_1,\cdots,j_n) \in \prod_{i=1}^{n} \Sigma_{N_i} ,k\in \Sigma_n.$
By the assumptions on $q_{j_1 j_2 \cdots j_n}$ and $v_{k,j_k}$, we can easily see that $T$ is well-defined. Now consider functions $f,g \in L^{q}_{0}(\mathcal I^{n},\mu_{p})$, then we have
\begin{align*}
    &\lVert Tf-Tg \rVert_{q}^{q} \\&=\int_{\mathcal I^{n}} \lvert (Tf)(y_1,y_2,\cdots,y_n)-(Tg)(y_1,y_2,\cdots,y_n) \rvert^{q} d\mu_{p}(y_1,y_2,\cdots,y_n)\\
    &= \sum_{(j_1,\cdots,j_n) \in \prod_{i=1}^{n} \Sigma_{N_i}} \int_{\prod_{k=1}^{n} I_{k,j_k}}\lvert \alpha(y_1,y_2,\cdots,y_n) f(v_{1,j_1}^{-1}(y_1),\cdots,v_{n,j_n}^{-1}(y_n))- \\ & \hspace{3cm} \alpha(y_1,y_2,\cdots,y_n) g(v_{1,j_1}^{-1}(y_1),\cdots,v_{n,jn}^{-1}(y_n)) \rvert^q d\mu_{p}(y_1,y_2,\cdots,y_n)\\
    &=\sum_{(j_1,\cdots,j_n) \in \prod_{i=1}^{n} \Sigma_{N_i}} \int_{\prod_{k=1}^{n} I_{k,j_k}}\lvert \alpha(y_1,y_2,\cdots,y_n) (f-g)(v_{1,j_1}^{-1}(y_1),\cdots,v_{n,j_n}^{-1}(y_n)) \rvert^q \\ & \hspace{+5cm} \cdot d\mu_{p} (y_1,y_2,\cdots,y_n)\\
    &\leq \sum_{(j_1,\cdots,j_n) \in \prod_{i=1}^{n} \Sigma_{N_i}} \lVert \alpha \rVert_{\infty}^{q}\int_{\prod_{k=1}^{n} I_{k,j_k}} \lvert (f-g)(v_{1,j_1}^{-1}(y_1),\cdots,v_{n,j_n}^{-1}(y_n))\rvert^q \\& \hspace{+5cm} \cdot d\mu_{p}(y_1,y_2,\cdots,y_n)\\
    \end{align*}
    \begin{align*}
    &= \sum_{(j_1,\cdots,j_n) \in \prod_{i=1}^{n} \Sigma_{N_i}} \lVert \alpha \rVert_{\infty}^{q}\int_{\prod_{k=1}^{n} I_{k,j_k}} \lvert (f-g)(v_{1,j_1}^{-1}(y_1),\cdots,v_{n,j_n}^{-1}(y_n))\rvert^q \\& \cdot  d\left(\sum_{(i_1,\cdots ,i_n)\in \prod_{k=1}^{n} \Sigma_{N_k}} p_{i_1\cdots i_n}\mu_{p} (v_{1,i_1}^{-1}(y_1),\cdots,v_{n,i_n}^{-1}(y_n))\right) \\
    &= \sum_{(j_1,\cdots,j_n) \in \prod_{i=1}^{n} \Sigma_{N_i}} \lVert \alpha \rVert_{\infty}^{q} \int_{I_{\prod_{k=1}^{n} I_{k,j_k}}}\lvert (f-g)(v_{1,j_1}^{-1}(y_1),\cdots,v_{n,j_n}^{-1}(y_n))\rvert^q \\ & \hspace{+5cm} \cdot p_{j_1 j_2 \cdots j_n} d(\mu_{p} (v_{1,j_1}^{-1}(y_1),\cdots,v_{n,j_n}^{-1}(y_n)))\\
    &= \sum_{(j_1,\cdots,j_n) \in \prod_{i=1}^{n} \Sigma_{N_i}} \lVert \alpha \rVert_{\infty}^{q} p_{j_1j_2 \cdots j_n} \int_{\mathcal I^{n}} \lvert (f-g)(y_1,y_2,\cdots,y_n) \rvert^q d\mu_{p} (y_1,y_2,\cdots,y_n)\\
    &= \sum_{(j_1,\cdots,j_n) \in \prod_{i=1}^{n} \Sigma_{N_i}} \lVert \alpha \rVert_{\infty}^{q}   p_{j_1j_2 \cdots j_n} \lVert f-g \rVert_{q}^{q}.
    \end{align*} 
    This implies that \[ \lVert Tf-Tg \rVert_{q}^{q} \leq \lVert \alpha \rVert_{\infty}^{q}  
\sum_{(j_1,\cdots,j_n) \in \prod_{i=1}^{n} \Sigma_{N_i}}  
p_{j_1 j_2 \cdots j_n} \lVert f-g \rVert_{q}^{q} \], since \[\sum_{(j_1,\cdots,j_n) \in \prod_{i=1}^{n} \Sigma_{N_i}}  p_{j_1 j_2 \cdots j_n}=1\] which further implies that $$\lVert Tf-Tg \rVert_{q}\leq \lVert \alpha \rVert_{\infty}\lVert f-g \rVert_{q},$$ that is, $T$  becomes a contraction. So, by Banach fixed point theorem,  $T$ have a unique fixed point, say, $h^*$ such that 
\begin{equation*}
\begin{split}
    h^*(y_1,y_2,\cdots,y_n)=&\alpha(y_1,y_2,\cdots,y_n) h^*(v_{1,j_1}^{-1}(y_1),\cdots,v_{n,j_n}^{-1}(y_n))\\&+q_{j_1j_2 \cdots j_n}(v_{1,j_1}^{-1}(y_1),\cdots,v_{n,j_n}^{-1}(y_n)).
    \end{split}
\end{equation*}
for all $(y_1,y_2,\cdots y_n)\in \prod_{k=1}^{n} I_{k,j_k}$, where$(j_1,\cdots,j_n)\in\prod_{i=1}^{n}\Sigma_{N_i}$.\\
Hence the proof.
\end{proof}
%\begin{remark}
%In \cite{massopust2016local}, Massopust has discussed local FIFs in $L^{p}$ spaces. There the measure is the one-dimensional Lebesgue measure. Here, we have proved the result for all invariant measures generated by IFS, so our result may be treated as a generalization of the previous result.
%\end{remark}
Next, the idea and characteristics of the $\alpha$-fractal operator $\mathcal{F}_{\Delta, L}^\alpha$, a linear operator related to fractal functions and Lebesgue spaces, are discussed. In addition, germ functions, linear operators, and scaling functions ($\alpha$) that meet particular requirements are introduced.
\subsection{ Associated Fractal Operator and its Properties}
 Consider a continuous function $f$ on $\mathcal I^{n}$ , which satisfy the initial data set and termed as germ function. Also, consider a bounded linear operator \( L:\mathcal{L}^q(\mathcal I^{n}, \mu_p)\to\mathcal{L}^q(\mathcal I^{n}, \mu_p) \)  such that for all $h\in \mathcal{L}^q(\mathcal I^{n}, \mu_p),~ Lh(y_1,y_2,\cdots ,y_n)=h(y_1,y_2,\cdots ,y_n), \text{ where } (y_1,y_2,\cdots ,y_n) \in \prod_{i=1}^{n}\partial\Sigma_{N_i,0}.$ Choose functions $F_{j_1 j_2 \cdots j_n}, \text{ for } (j_1,\cdots ,j_n)\in \prod_{i=1}^{n} \Sigma_{N_i}.$ as 
 \begin{equation*}
   \begin{split}
       F_{j_1\cdots j_n}(y_1,y_2, \cdots ,y_n,x)=&\alpha(v_{1,j_1}(y_1),\cdots,v_{n,j_n}(y_n))x+f(v_{1,j_1}(y_1),\cdots,v_{n,j_n}(y_n))-\\&\alpha(v_{1,j_1}(y_1),\cdots,v_{n,j_n}(y_n))(Lf)(y_1,y_2,\cdots ,y_n),
   \end{split} 
\end{equation*}
where $(y_1,y_2, \cdots ,y_n,x)\in M$. Here, the functions $\alpha \in \mathcal{L}^{\infty}(\mathcal I^{n}, \mu_p)$ satisfy the condition $\|\alpha\|_{\infty}<1$  and this function  $\alpha$ is termed as the scaling function. The fixed point of the RB operator is $f^{\alpha}$ which is an  $\alpha$-FIF. Corresponding to this $\alpha$-FIF, the RB operator is described as follows:
% Using Theorem 3.1 of \cite{SP}
\begin{equation}\label{D3}
    \begin{split}
    T^{\alpha}(g)(y_1,y_2, \cdots ,y_n)=&f(y_1,y_2, \cdots ,y_n)+\alpha(y_1,y_2, \cdots ,y_n)(g(v_{1,j_1}^{-1}(y_1),\cdots,v_{n,j_n}^{-1}(y_n))\\&
    -Lf(v_{1,j_1}^{-1}(y_1),\cdots,v_{n,j_n}^{-1}(y_n))),
    \end{split}
\end{equation}
for $(y_1,y_2, \cdots ,y_n)\in \prod_{k=1}^{n}I_{k,j_k}$ and $(j_1,j_2,\cdots ,j_n)\in \prod_{i=1}^{n} \Sigma_{N_i},$ and obeys the self-referential equation as
\begin{equation}\label{D4}
    \begin{aligned}
    f^{\alpha}(y_1,y_2, \cdots ,y_n)=&f(y_1,y_2, \cdots ,y_n)+\alpha(y_1,y_2, \cdots ,y_n) (f^{\alpha}(v_{1,j_1}^{-1}(y_1),\cdots,v_{n,j_n}^{-1}(y_n))\\&-(Lf)(v_{1,j_1}^{-1}(y_1),\cdots,v_{n,j_n}^{-1}(y_n))),
    \end{aligned}
\end{equation}
for $(y_1,y_2, \cdots ,y_n)\in \prod_{k=1}^{n}I_{k,j_k}$ and $(j_1,j_2,\cdots ,j_n)\in \prod_{i=1}^{n} \Sigma_{N_i}.$\\
Consequently, there exists a bounded linear operator $\mathcal{F}_{\Delta, L}^\alpha:\mathcal{L}^q(\mathcal I^{n}, \mu_p) \to \mathcal{L}^q(\mathcal I^{n}, \mu_p)$, defined as $\mathcal{F}_{\Delta, L}^\alpha(f)=f^{\alpha} \text{ for } f\in \mathcal{L}^q(\mathcal I^{n}, \mu_p).$   corresponding to the net $\Delta$, a base operator $L$, and a germ function $f$ , the operator $\mathcal{F}_{\Delta, L}^\alpha$ is termed as $\alpha$-fractal operator or simply fractal operator.\\
\newline

\begin{theorem} \label{thm3}
Let $Id \in \mathcal{L}^q(\mathcal I^{n}, \mu_p)$ be the   the identity operator.  We get the following results:
\begin{enumerate}
    \item If $f \in \mathcal{L}^q(\mathcal I^{n}, \mu_p)$, then the perturbation error satisfies the  relation:
$$
\left\|f^\alpha-f\right\|_{q} \leq \frac{\|\alpha\|_{\infty}}{1-\|\alpha\|_{\infty}}\|f-L f\|_{q} .
$$
In the particular case $\alpha=0$, we get $\mathcal{F}_{\Delta, L}^\alpha=I d$.

\item The fractal operator $\mathcal{F}_{\Delta, L}^\alpha$ is a bounded linear operator  on the space $\mathcal{L}^q(\mathcal I^{n}, \mu_p)$ and  we obtain the following estimate on its  operator norm 
$$
\left\|\mathcal{F}_{\Delta, L}^\alpha\right\| \leq 1+\frac{\|\alpha\|_{\infty}\|I d-L\|}{1-\|\alpha\|_{\infty}}
$$

\item  Under the condition $\|\alpha\|_{\infty}<\|L\|^{-1},$ $  \mathcal{F}_{\Delta, L}^\alpha$ is bounded below and one to one.
\item \label{eqn34} Under the condition $\|\alpha\|_{\infty}<(1+\|I d-L\|)^{-1}$,  $\mathcal{F}_{\Delta, L}^\alpha$ has a bounded inverse and as a result,  a topological automorphism . We also have,
$$
\left\|\left(\mathcal{F}_{\Delta, L}^\alpha\right)^{-1}\right\| \leq \frac{1+\|\alpha\|_{\infty}}{1-\|\alpha\|_{\infty}\|L\|} .
$$ 
\item The fixed points of $L$ and the $\mathcal{F}_{\Delta, L}^\alpha$ are same, whenever $\|\alpha\|_{\infty} \neq 0.$
\item  The fractal operator norm satisfies $1 \leq\left\|\mathcal{F}_{\Delta, L}^\alpha\right\|,$ whenever 1 belongs to the point spectrum of $L.$
\item Under the condition $\|\alpha\|_{\infty}<\|L\|^{-1}$, the operator $\mathcal{F}_{\Delta, L}^\alpha$ fails to be  a compact operator.
\item Assuming the condition $\|\alpha\|_{\infty}<(1+\|I d-L\|)^{-1}$,  $\mathcal{F}_{\Delta, L}^\alpha$  is a Fredholm operator with an index of zero.
\end{enumerate}
\end{theorem}
\begin{proof}
\begin{enumerate}
\item By the use of self-referential equation, we can get
\begin{align*}
    &\lVert f^\alpha-f \rVert_{q}^{q} \\&=\int_{\mathcal I^{n}} \lvert f^\alpha(y_1,y_2,\cdots,y_n)-f(y_1,y_2,\cdots,y_n) \rvert^{q} d\mu_{p}(y_1,y_2,\cdots,y_n)\\
    &= \sum_{(j_1,\cdots,j_n) \in \prod_{i=1}^{n} \Sigma_{N_i}} \int_{\prod_{k=1}^{n} I_{k,j_k}}\lvert \alpha(y_1,y_2,\cdots,y_n) f^\alpha(v_{1,j_1}^{-1}(y_1),\cdots,v_{n,j_n}^{-1}(y_n))- \\& \hspace{+2cm}Lf(v_{1,j_1}^{-1}(y_1),\cdots,v_{n,j_n}^{-1}(y_n)) \rvert^q d\mu_{p}(y_1,y_2,\cdots,y_n)\\
    &=\sum_{(j_1,\cdots,j_n) \in \prod_{i=1}^{n} \Sigma_{N_i}} \int_{\prod_{k=1}^{n} I_{k,j_k}}\lvert \alpha(y_1,y_2,\cdots,y_n) (f^\alpha-Lf)\\&\hspace{2.5cm}(v_{1,j_1}^{-1}(y_1),\cdots,v_{n,j_n}^{-1}(y_n)) \rvert^q d\mu_{p} (y_1,y_2,\cdots,y_n)\\
   \end{align*} 
   \begin{align*}
    &\leq \sum_{(j_1,\cdots,j_n) \in \prod_{i=1}^{n} \Sigma_{N_i}} \lVert \alpha \rVert_{\infty}^{q}\int_{\prod_{k=1}^{n} I_{k,j_k}} \lvert (f^\alpha-Lf)(v_{1,j_1}^{-1}(y_1),\cdots,v_{n,j_n}^{-1}(y_n))\rvert^q \\& \hspace{+5cm}\cdot d\mu_{p}(y_1,y_2,\cdots,y_n)\\
    &= \sum_{(j_1,\cdots,j_n) \in \prod_{i=1}^{n} \Sigma_{N_i}} \lVert \alpha \rVert_{\infty}^{q}\int_{\prod_{k=1}^{n} I_{k,j_k}} \lvert (f^\alpha-Lf)(v_{1,j_1}^{-1}(y_1),\cdots,v_{n,j_n}^{-1}(y_n))\rvert^q \\& \cdot  d\left(\sum_{(i_1,\cdots,i_n)\in \prod_{k=1}^n\Sigma_{N_k}} p_{i_1\cdots i_n}\mu_{p} (v_{1,i_1}^{-1}(y_1),\cdots,v_{n,i_n}^{-1}(y_n))\right) \\
    &= \sum_{(j_1,\cdots,j_n) \in \prod_{i=1}^{n} \Sigma_{N_i}} \lVert \alpha \rVert_{\infty}^{q} \int_{\prod_{k=1}^{n} I_{k,j_k}}\lvert (f^\alpha-Lf)(v_{1,j_1}^{-1}(y_1),\cdots,v_{n,j_n}^{-1}(y_n))\rvert^q\\& \hspace{+5cm} p_{j_1\cdots j_n} d(\mu_{p} (v_{1,j_1}^{-1}(y_1),\cdots,v_{n,j_n}^{-1}(y_n)))\\
    &= \sum_{(j_1,\cdots,j_n) \in \prod_{i=1}^{n} \Sigma_{N_i}} \lVert \alpha \rVert_{\infty}^{q} p_{j_1\cdots j_n} \int_{\mathcal{I}^n} \lvert (f^\alpha-Lf)(y_1,y_2,\cdots,y_n) \rvert^q d\mu_{p} (y_1,y_2,\cdots,y_n)\\
    &=\sum_{(j_1,\cdots,j_n) \in \prod_{i=1}^{n} \Sigma_{N_i}} \lVert \alpha \rVert_{\infty}^{q} p_{j_1 j_2\cdots j_n} \lVert f^\alpha-Lf \rVert_{q}^{q}.
    \end{align*}
    From here ,we get
    \begin{equation}\label{eqn31}
    \lVert f^\alpha-f \rVert_{q}\leq \lVert \alpha \rVert_{\infty} \lVert f^\alpha-Lf \rVert_{q},
    \end{equation}
    By triangle inequality, we have
    $$\lVert f^\alpha-f \rVert_{q}\leq \lVert \alpha \rVert_{\infty} \lVert f^\alpha-f \rVert_{q} + \lVert \alpha \rVert_{\infty} \lVert f -Lf \rVert_{q},$$
    Hence, we get the desired inequality
    \begin{equation}\label{eqn312}
\left\|f^\alpha-f\right\|_{q} \leq \frac{\|\alpha\|_{\infty}}{1-\|\alpha\|_{\infty}}\|f-L f\|_{q} .
\end{equation}

\item By using inequality (\ref{eqn312}), we get

\label{eqn32}
\begin{align*}
\left\|f^\alpha\right\|_{q}- \left\| f \right\|_{q}& \leq \left\|f^\alpha-f\right\|_{q}\\  &\leq \Bigg[\frac{\|\alpha\|_{\infty}\|I d-L\|}{1-\|\alpha\|_{\infty}} \Bigg]\left\| f \right\|_{q} 
\end{align*}
Hence, we get the desired inequality $$ \left\|\mathcal{F}_{\Delta, L}^\alpha \right\| \leq  1 + \frac{\|\alpha\|_{\infty}\|I d-L\|}{1-\|\alpha\|_{\infty}}   $$

\item 
\begin{align*}
\left\|f\right\|_{q}- \left\| f^\alpha \right\|_{q}& \leq \left\|f^\alpha-f\right\|_{q} \\& \leq \|\alpha\|_{\infty} \left\|f^\alpha-Lf\right\|_{q}
\\& \leq \|\alpha\|_{\infty} \big (\left\| f^\alpha \right\| _{q} + \|L\| \left\|f\right\|_{q}  \big)
\end{align*}

This gives
$$ \big(1- \|\alpha\|_{\infty}\|L\| \big)\|f\|_{q} \leq \big(1+\|\alpha\|_{\infty} \big)\| f^\alpha \|_{q}$$
By assuming $\|\alpha\|_{\infty} < \|L\|^{-1}$, we get
\begin{equation}\label{eqn33}
 \|f\|_{q}  \leq \frac{1+\|\alpha\|_{\infty}}{1-\|\alpha\|_{\infty}\|L\|}\|f^\alpha\|_{q}  
\end{equation}
Hence, we get the required result.
\item  By the use of inequality (\ref{eqn32}), we get $$ \left\| Id-\mathcal{F}_{\Delta, L}^\alpha \right\| \leq   \frac{\|\alpha\|_{\infty}\|I d-L\|}{1-\|\alpha\|_{\infty}}   $$
Now, by using the given condition, we get $\left\| Id-\mathcal{F}_{\Delta, L}^\alpha \right\| \leq 1$, therefore $(Id-\mathcal{F}_{\Delta, L}^\alpha )^{-1}$ exists, and this inverse is bounded. By using inequality (\ref{eqn33}), we can get $$\left\| (\mathcal{F}_{\Delta, L}^\alpha)^{-1}(f)\right\|_{q}  \leq \frac{1+\|\alpha\|_{\infty}}{1-\|\alpha\|_{\infty}\|L\|}\|f^\alpha\|_{q} $$
This gives the required result.
\item Consider a fixed point of $L$ namely $f$ and  $\|\alpha\|_{\infty} \neq 0.$ By using the condition on 
 $\alpha$ and the inequality (\ref{eqn31}), we get $ \left\|f^\alpha-f\right\|_{q}=0.$ Therefore, by using the property of norm, we get $f^\alpha=f$ almost everywhere. This gives that $f^\alpha=f$ in the given space $\mathcal{L}^q(\mathcal{I}^n, \mu_p).$
\item Consider an element $f$ of the space  $\mathcal{L}^q(\mathcal{I}^n, \mu_p)$ with the imposing condition $L(f)=f$ and $\|f\|_{q}=1.$ By using above part, one can easily deduce that $\| \mathcal{F}_{\Delta, L}^\alpha(f)\|_{q}=\|f\|_{q}. $ Also directly from the definition of operator norm, one can get $1 \leq \| \mathcal{F}_{\Delta, L}^\alpha\|.$ 
\item We already proved in part $3$, that $\mathcal{F}_{\Delta, L}^\alpha$ is one to one. Since the range space of $\mathcal{F}_{\Delta, L}^\alpha$ has dimension infinite. Next, we have to define an inverse map $(\mathcal{F}_{\Delta, L}^\alpha)^{-1}: \mathcal{F}_{\Delta, L}^\alpha(\mathcal{L}^q(\mathcal{I}^n, \mu_p) )\rightarrow \mathcal{L}^q(\mathcal{I}^n, \mu_p).$ By using the condition $\|\alpha\|_{\infty}<1$, we get $\mathcal{F}_{\Delta, L}^\alpha$ is bounded below. From here it can be directly seen that $(\mathcal{F}_{\Delta, L}^\alpha)^{-1}$ is a bounded linear operator. Let us assume that $\mathcal{F}_{\Delta, L}^\alpha$ is a compact operator. Then $\mathcal{F}_{\Delta, L}^\alpha(\mathcal{F}_{\Delta, L}^\alpha)^{-1}$ becomes a compact operator. Since range space of $\mathcal{F}_{\Delta, L}^\alpha$ has infinite dimension, which is a contradiction to the previous statement. Hence the proof.
\item By using the given assumption, $(\mathcal{F}_{\Delta, L}^\alpha)^{*}$ is invertible. Therefore, $(\mathcal{F}_{\Delta, L}^\alpha)$ is Fredholm. The index of a Fredholm operator is determined by $$ index (\mathcal{F}_{\Delta, L}^\alpha)=\dim (kernel(\mathcal{F}_{\Delta, L}^\alpha))-\dim (kernel(\mathcal{F}_{\Delta, L}^\alpha)^{*}).$$ Hence the index for a Fredholm operator is zero.
\end{enumerate}
\end{proof}

\begin{theorem} Consider a set  $\mathcal{B}=\{ \alpha \in L^{\infty}(\mathcal{I}^n,\mu_p): \lVert \alpha \rVert_{\infty} < 1\}$. Then  \( \mathcal{T}:\mathcal{B} \to\mathbb{R} \) defined by $\mathcal{T}(\alpha) = \| T^{\alpha}(0) \|_{q}$ is convex, where $T^{\alpha}$ is a RB operator mentioned in equation (\ref{D3}).
\end{theorem}
\begin{proof}
Define  \( t_{\alpha}:\mathcal{I}^n \to\mathbb{R} \) as follows $$t_{\alpha}(y_1,y_2,\cdots,y_n)=\alpha(y_1,y_2,\cdots,y_n)Lf(v_{1,j_1}(y_1),\cdots,v_{n,j_n}(y_n))$$
for $(y_1,y_2,\cdots,y_n)\in \prod_{k=1}^{n} I_{k,j_k}$ and $(j_1,j_2,\cdots,j_n)\in \prod_{i=1}^{n}\Sigma_{N_i} .$ Let $0<c<1$ and $\alpha_{1},\alpha_{2}$ in $\mathcal{B}.$
\begin{align*}
\mathcal{T}(c\alpha_{1}+(1-c)\alpha_{2})&= \| T^{c\alpha_{1}+(1-c)\alpha_{2}}(0) \|_{q}
\\&= \|f-t_{c\alpha_{1}+(1-c)\alpha_{2}} \|_{q}
\\&= \|f-(ct_{\alpha_{1}}+ (1-c)t_{\alpha_{2}}\|_{q}
\\&\leq c\|f-t_{\alpha_{1}}\|_{q}+ (1-c)\|f-t_{\alpha_{2}}\|_{q}
\\&= c \mathcal{T}(\alpha_{1})+(1-c)\mathcal{T}(\alpha_{2})
\end{align*}
Therefore $\mathcal{T}$ is a convex map.

\end{proof}
The constraint of linearity for the operator $L$ can be dropped in the next statement since we treat $L$ as an affine map.
\begin{theorem} Consider a fractal operator $\mathcal{F}_{\Delta, L}^\alpha: \mathcal{L}^q(\mathcal{I}^n , \mu_p) \rightarrow \mathcal{L}^q(\mathcal{I}^n , \mu_p)$ . Prove that it is relatively Lipschitz  with respect to the base operator $L.$
\end{theorem}
\begin{proof}
Consider $ f_1, f_2 \in \mathcal{L}^q(\mathcal{I}^n, \mu_p) $ and corresponding to these $ f_{1}^{\alpha} $ and $ f_{2}^{\alpha} $ become the fixed points of the RB operator $T_{f_{1}}^{\alpha}$ and $T_{f_{2}}^{\alpha}$ respectively. The RB operators $T_{f_{1}}^{\alpha}$ and $T_{f_{2}}^{\alpha}$ can be defined as

\begin{equation*}
    \begin{aligned}
    T_{f_{1}}^{\alpha}(g)(y_1,y_2,\cdots,y_n)=f_{1}(y_1,y_2,\cdots,y_n)+\alpha(y_1,y_2,\cdots,y_n)\times\\(g(v_{1,j_1}^{-1}(y_1),\cdots,v_{n,j_n}^{-1}(y_n))-Lf_{1}(v_{1,j_1}^{-1}(y_1),\cdots,v_{n,j_n}^{-1}(y_n))),\\
    T_{f_{2}}^{\alpha}(g)(y_1,y_2,\cdots,y_n)=f_{2}(y_1,y_2,\cdots,y_n)+\alpha(y_1,y_2,\cdots,y_n)\times\\(g(v_{1,j_1}^{-1}(y_1),\cdots,v_{n,j_n}^{-1}(y_n))-Lf_{2}(v_{1,j_1}^{-1}(y_1),\cdots,v_{n,j_n}^{-1}(y_n))),
    \end{aligned}
\end{equation*}
for $(y_1,y_2,\cdots,y_n)\in \prod_{k=1}^n I_{k,j_k}$ and $(j_1,j_2,\cdots,j_n)\in \prod_{i=1}^n \Sigma_{N_i}.$ 
\begin{align*}
    &\lVert T_{f_{1}}^{\alpha}(g)-T_{f_{2}}^{\alpha}(g) \rVert_{q} \\&\leq \left\|f_{1}-f_{2}\right\|_{q}+\Bigg( \int_{\mathcal{I}^n} \lvert \alpha(y_1,y_2,\cdots,y_n)(L(f_1)-L(f_2))(v_{1,j_1}^{-1}(y_1),\cdots,v_{n,j_n}^{-1}(y_n))\rvert^{q} \\& \hspace{9cm}\cdot d\mu_{p}(y_1,y_2,\cdots,y_n)\Bigg)^{\frac{1}{q}}\\
    &= \left\|f_{1}-f_{2}\right\|_{q}+ \Bigg( \sum_{(j_1,j_2,\cdots,j_n)\in \prod_{i=1}^n \Sigma_{N_i}} \int_{\prod_{k=1}^{n}I_{k,j_k}}\lvert \alpha(y_1,y_2,\cdots,y_n) \\& \hspace{2.8cm}\cdot(L(f_1)-L(f_2))(v_{1,j_1}^{-1}(y_1),\cdots,v_{n,j_n}^{-1}(y_n)) \rvert^q d\mu_{p} (y_1,y_2,\cdots,y_n)\Bigg)^{\frac{1}{q}}\\
    &\leq \left\|f_{1}-f_{2}\right\|_{q} + \Bigg(\sum_{(j_1,j_2,\cdots,j_n)\in \prod_{i=1}^n \Sigma_{N_i}} \lVert \alpha \rVert_{\infty}^{q}\int_{\prod_{k=1}^{n}I_{k,j_k}} \lvert (L(f_1)-L(f_2)) \\& (v_{1,j_1}^{-1}(y_1),\cdots,v_{n,j_n}^{-1}(y_n))\rvert^q d\mu_{p}(y_1,y_2,\cdots,y_n)\Bigg)^{\frac{1}{q}}\\
      &=  \left\|f_{1}-f_{2}\right\|_{q}+ \Bigg(\sum_{(j_1,j_2,\cdots,j_n)\in \prod_{i=1}^n \Sigma_{N_i}} \lVert \alpha \rVert_{\infty}^{q}\int_{\prod_{k=1}^{n}I_{k,j_k}} \lvert (L(f_1)-L(f_2)) \\& (v_{1,j_1}^{-1}(y_1),\cdots,v_{n,j_n}^{-1}(y_n))\rvert^q \\& \cdot  d\Bigg(\sum_{(i_1,i_2,\cdots,i_n)\in \prod_{k=1}^n \Sigma_{N_k}} p_{i_1 i_2\cdots i_n}\mu_{p} (v_{1,i_1}^{-1}(y_1),\cdots,v_{n,i_n}^{-1}(y_n)) \Bigg)\Bigg)^{\frac{1}{q}}  \\
    &=\left\|f_{1}-f_{2}\right\|_{q}+ \bigg(\sum_{(j_1,j_2,\cdots,j_n)\in \prod_{i=1}^n \Sigma_{N_i}} \lVert \alpha \rVert_{\infty}^{q} \int_{\prod_{k=1}^{n}I_{k,j_k}}\lvert (L(f_1)-L(f_2))\\&(v_{1,j_1}^{-1}(y_1),\cdots,v_{n,j_n}^{-1}(y_n))\rvert^q p_{j_1\cdots j_n} d(\mu_{p} (u_{i}^{-1}(x),v_{j}^{-1}(y))\Bigg)^{\frac{1}{q}} \\
    &= \left\|f_{1}-f_{2}\right\|_{q}+ \left( \sum_{(j_1,j_2,\cdots,j_n)\in \prod_{i=1}^n \Sigma_{N_i}} \lVert \alpha \rVert_{\infty}^{q} p_{j_1\cdots j_n} \int_{\mathcal{I}^n} \lvert (L(f_1)-L(f_2))(x,y) \rvert^q d\mu_{p} (x,y)\right)^{\frac{1}{q}}\\
    &=  \left\|f_{1}-f_{2}\right\|_{q}+ \left( \sum_{(j_1,j_2,\cdots,j_n)\in \prod_{i=1}^n \Sigma_{N_i}} \lVert \alpha \rVert_{\infty}^{q} p_{j_1\cdots j_n} \lVert L(f_1)-L(f_2) \rVert_{q}^{q}\right)^{\frac{1}{q}}\\
    &= \left\|f_{1}-f_{2}\right\|_{q}+ \lVert \alpha \rVert_{\infty}\lVert L(f_1)-L(f_2) \rVert_{q}.
\end{align*}
From here, we get
\begin{align*}
    \lVert \mathcal{F}_{\Delta, L}^\alpha (f_{1})-\mathcal{F}_{\Delta, L}^\alpha (f_{2}) \rVert_{q} 
    &= \lVert f_{1}^{\alpha}-f_{2}^{\alpha}\rVert_{q} \\
    &=\lVert T_{f_{1}}^{\alpha}(f_{1}^{\alpha})-T_{f_{2}}^{\alpha}(f_{2}^{\alpha}) \rVert_{q} \\
    &\leq \lVert T_{f_{1}}^{\alpha}(f_{1}^{\alpha})-T_{f_{2}}^{\alpha}(f_{1}^{\alpha}) \rVert_{q}  +\lVert T_{f_{2}}^{\alpha}(f_{1}^{\alpha})-T_{f_{2}}^{\alpha}(f_{2}^{\alpha}) \rVert_{q} \\
    \end{align*}
  \begin{align*}
    &\leq \left\|f_{1}-f_{2}\right\|_{q}+ \lVert \alpha \rVert_{\infty}\lVert L(f_1)-L(f_2) \rVert_{q} + \lVert \alpha \rVert_{\infty}\lVert  f_{1}^{\alpha}-f_{2}^{\alpha} \rVert_{q}
\end{align*}
This produces
$$\lVert \mathcal{F}_{\Delta, L}^\alpha (f_{1})-\mathcal{F}_{\Delta, L}^\alpha (f_{2}) \rVert_{q} \leq  \frac{1}{1-\|\alpha\|_{\infty}}\lVert f_1-f_2 \rVert_{q} +\frac{\|\alpha\|_{\infty}}{1-\|\alpha\|_{\infty}}\lVert L(f_1)-L(f_2) \rVert_{q} $$
Hence the proof.
\end{proof}
%\newline

Next, we explore fractal interpolation by using a specific germ function and a base function. Then, we generate fractal images at different scaling factors where varying demonstrates how the fractal structure changes with the scaling factor.\\
 Consider a domain $[-1,1] \times [-1,1]$ of rectangular type and  $\{-1,-0.5,~0.0,~0.5,~1\}$ is a partition corresponding to $[-1,1]$. Now, we have to consider a germ function in the specific domain as follows: 
$$f(x,y)=(x^2+y^2)\sin \left( \frac{1}{\sqrt{1 + x^2 + y^2}} \right).$$
\newline

Choose the base function as $$ L(x,y)=(2-x^2)y^2f(x,y), $$
for all $(x,y) \in [-1,1] \times [-1,1].$
Then for various scaling factors $\alpha=0.3,~0.5,~ 0.7,~0.9$ we obtain the figures $1,~2,~3,~4$, respectively.\\
\begin{figure}[!h]
\begin{minipage}[t]{6.5cm}
\includegraphics[scale=0.4]{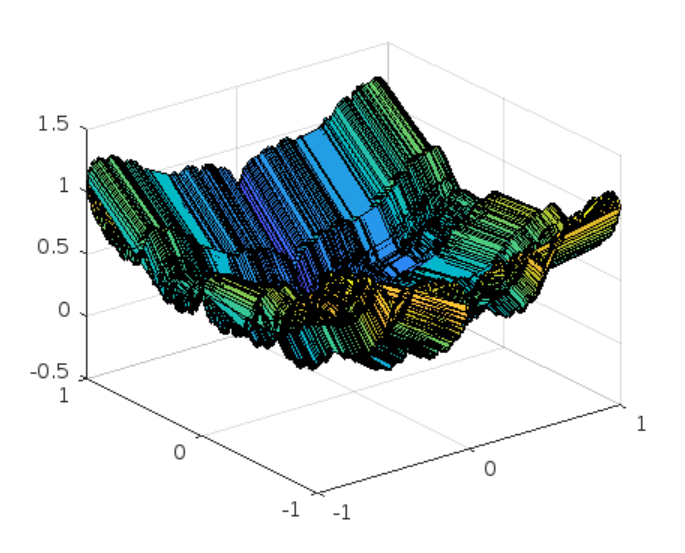}
\caption{FIS($\alpha=0.3$)}
\end{minipage}
\hspace{1.5cm}
\begin{minipage}[t]{6.5cm}
\includegraphics[scale=0.4]{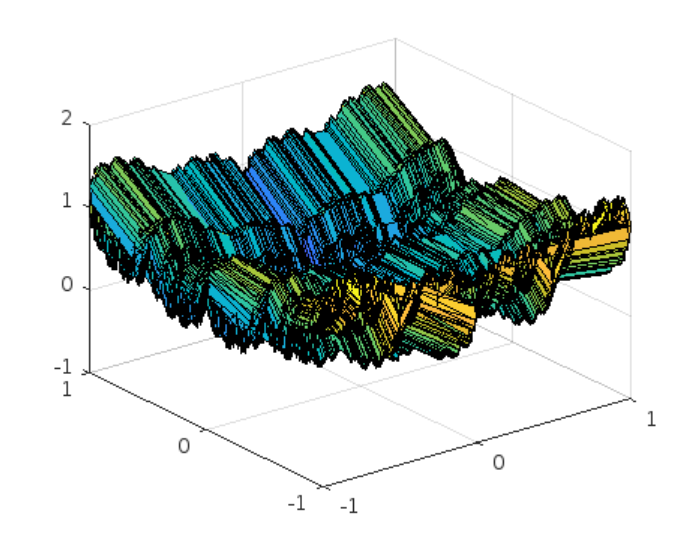}
\caption{FIS($\alpha=0.5$)}
\end{minipage}
\end{figure}
\begin{figure}[!h]
\begin{minipage}[t]{6.5cm}
\includegraphics[scale=0.4]{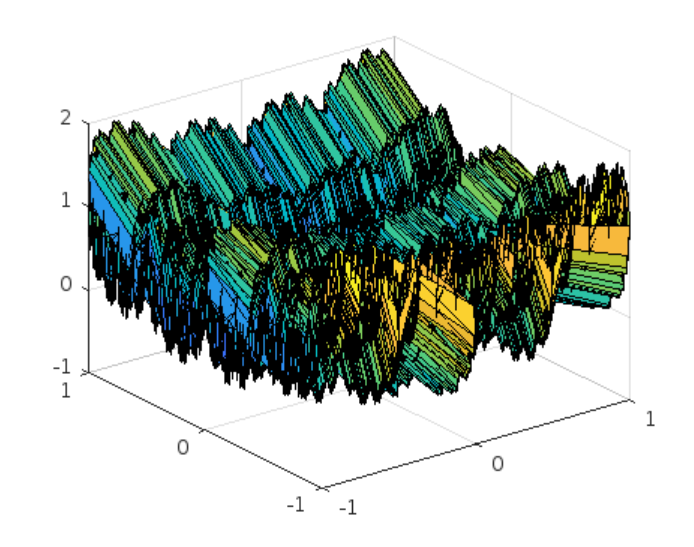}
\caption{FIS($\alpha=0.7$)}
\end{minipage}
\hspace{1.5cm}
\begin{minipage}[t]{6.5cm}
\includegraphics[scale=0.4]{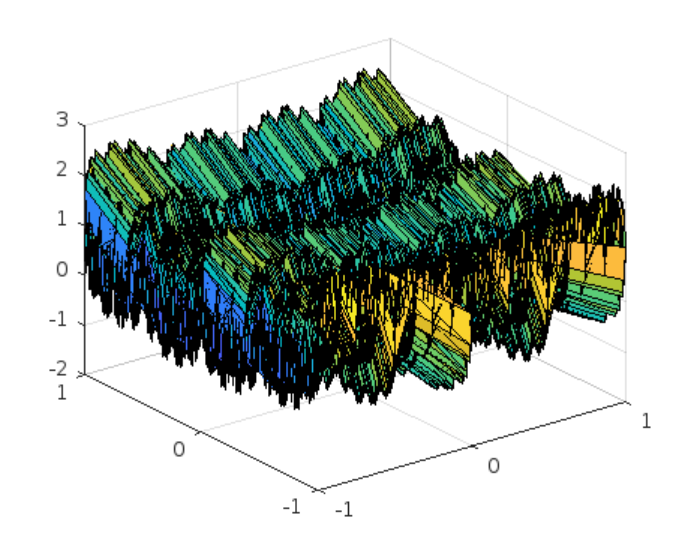}
\caption{FIS($\alpha =0.9$)}
\end{minipage}
\end{figure}
\newline

A Schauder basis is a mathematical notion used in the field of functional analysis. This idea is essential to the study of infinite-dimensional spaces and can be used in many branches of applied sciences and mathematics. Schonefeld \cite{SQ1} introduced the Schauder basis concept. The necessity to extend the notion of a basis from finite-dimensional vector spaces to infinite-dimensional Banach spaces has ultimately led to the development of Schauder bases. He proved that some special functions ${h_n}$ form a basis for the space of one time continuously differentiable on $[0~ 1] \times [0~ 1] $. Later in 1971,\cite{SQ2} generalized this concept of Schauder bases for the space $C^k(T^q)$. This motivates us to discuss the Schauder basis for the space $\mathcal{L}^q(\mathcal{I}^n, \mu_p).$ \\ We address whether the corresponding fractal functions for the space $\mathcal{L}^q(\mathcal{I}^n, \mu_p)$ have a Schauder basis in the next theorem.
\begin{theorem} In the space $\mathcal{L}^q(\mathcal{I}^n, \mu_p)$ ,there exists a Schauder basis of the fractal functions corresponding to a given Schauder basis.
\end{theorem}
\begin{proof}
 Let us assume that $(h_{n})$ is a given Schauder basis for the space $\mathcal{L}^q(\mathcal{I}^n, \mu_p).$ With the help of a chosen $\alpha$ that satisfies condition $\|\alpha\|_{\infty}<(1+\|I d-L\|)^{-1}$ and part \ref{eqn34} of Theorem \ref{thm3}, it can be directly seen that $\mathcal{F}_{\Delta, L}^\alpha$ is a topological automorphism. If $ g\in \mathcal{L}^q(\mathcal{I}^n, \mu_p) $ then $(\mathcal{F}_{\Delta, L}^\alpha)^{-1}(g)\in \mathcal{L}^q(\mathcal{I}^n, \mu_p)$, with
 $$(\mathcal{F}_{\Delta, L}^\alpha)^{-1}(g)= \sum_{n=1}^{\infty} a_n \bigg((\mathcal{F}_{\Delta, L}^\alpha)^{-1}(g)\bigg)h_{n}.$$
 We already know that the fractal operator $\mathcal{F}_{\Delta, L}^\alpha$ is continuous. Therefore , we have
 $$g=\mathcal{F}_{\Delta, L}^\alpha(\mathcal{F}_{\Delta, L}^\alpha)^{-1}(g)=\sum_{n=1}^{\infty} a_{n} \bigg((\mathcal{F}_{\Delta, L}^\alpha)^{-1}(g)\bigg)h_{n}^{\alpha},$$
 where $h_{n}^{\alpha}=\mathcal{F}_{\Delta, L}^\alpha(h_{n})$. Let $g$  be represented as $g= \sum_{n=1}^{\infty} b_nh_{n}^{\alpha}.$ Since $(\mathcal{F}_{\Delta, L}^\alpha)^{-1}$ is also continuous, we have
 
   $$(\mathcal{F}_{\Delta, L}^\alpha)^{-1}(g)= \sum _{n=1}^{\infty}b_{n}h_{n},$$
   and it gives $b_n=a_n ((\mathcal{F}_{\Delta, L}^\alpha)^{-1}(g))$ for every $n$. Thus, we have a Schauder basis of the fractal functions $(h_{n}^{\alpha}) $ for $\mathcal{L}^q(\mathcal{I}^n, \mu_p)$. 
 \end{proof}
 
\section*{Declaration}
\noindent
\textbf{Conflicts of interest.} The authors declare no conflict of interest.\\
\\
\noindent
\textbf{Data availability:} Not applicable.\\
\\
\noindent
\textbf{Code availability:} Not applicable.\\
\\
\noindent
\textbf{Authors' contributions:} Each author contributed equally to this manuscript.

\section*{Acknowledgements} The first author is grateful to the government of India for providing financial support in the form of institutional fellowship.

\bibliographystyle{amsplain}

\begin{thebibliography}{10}
\bibitem{ES1} E. Agrawal, S. Verma, Fractal surfaces in H\"older and Sobolev spaces,  J Anal (2023) 1-19.
\bibitem{ES} E. Agrawal, S. Verma, Dimensions and Stability of invariant measures supported on Fractal Surfaces, Discrete and Continuous Dynamical Systems - S. doi: 10.3934/dcdss.2024159
\bibitem{VMT} V. Agrawal, M. Pandey, T. Som, Box Dimension and Fractional Integrals of Multivariate $\alpha -$Fractal Functions, Mediterr. J. Math. 20 (2023), no. 3, Paper No. 164, 23 pp.
\bibitem{MF1} M. F. Barnsley, Fractal functions and interpolation, Constr. Approx. 2 (1986) 303-329.
\bibitem {MF2} M. F. Barnsley, Fractal Everywhere, Academic Press, Orlando, Florida, 1988.
\bibitem{D2} P. Bouboulis, L. Dalla, V. Drakopoulos, Construction of recurrent bivariate fractal interpolation surfaces and computation of their box-counting dimension, J. Approx. Theory 141 (2006) 99-117.
\bibitem{D1} P. Bouboulis, L. Dalla, A general construction of fractal interpolation functions on grids of $\mathbb{R}^n$, Eur. J. Appl. Math. 18 (2007) 449-476.
\bibitem{SS4} S. Chandra, S. Abbas, The calculus of bivariate fractal interpolation surfaces, Fractals 29(03) (2021) 2150066.
\bibitem {Dalla} L. Dalla, Bivariate fractal interpolation functions on grids, Fractals 10 (2002) 53-58.
\bibitem{Fal} K. J. Falconer, Fractal Geometry: Mathematical Foundations and Applications, John Wiley Sons Inc., New York, 1999.
\bibitem{Feng} Z. Feng, Variation and Minkowski dimension of fractal interpolation surfaces, J. Math.
Anal. Appl., 345 1(2008) 322-334.

\bibitem{GH} J. S. Geronimo, D. Hardin, Fractal interpolation surfaces and a related 2-D multiresolution analysis, J. Math. Anal. Appl. 176(2) (1993) 561-586.
\bibitem{GVS} Gurubachan, V. V. M. S. Chandramouli, S. Verma, Fractal Dimension of $\alpha$-Fractal Functions Without Endpoint Conditions, Mediterr. J. Math. 21 (2024), no. 3, Paper No. 71, 23 pp.
\bibitem{KRZ} E. Kreyszig, Introductory Functional Analysis with Applications, John Wiley $\&$ Sons, Inc., New York. (1989)
\bibitem{RSP} R. Lal, S. Chandra, A. Prajapati, Fractal surfaces in Lebesgue spaces with respect to fractal measures and associated fractal operators, Chaos Solitons Fractals 181 (2024), 114684, 7 pp.
 \bibitem{Mz} R. Małysz, The Minkowski dimension of the bivariate fractal interpolation surfaces, Chaos, Solitons \& Fractals 27(5) (2006) 1147-1156.
 
 \bibitem{Mas4} P. R. Massopust, Fractal surfaces, Journal of Mathematical Analysis and Applications 151(1) (1990) 275-290. 


\bibitem{M21} M. A. Navascu$\acute{e}$s, Fractal polynomial interpolation, Z. Anal. Anwend. 25(2) (2005) 401-418.
\bibitem{M2} M. A. Navascu$\acute{e}$s, Fractal approximation, Complex Anal. Oper. Theory 4(4) (2010) 953-974.
\bibitem{NMA} M. A. Navascu$\acute{e}$s, R. N. Mohapatra, M. N. Akhtar, Construction of fractal surfaces, Fractals 28(02) (2020) 2050033.

\bibitem {NVV2}  M. A. Navascu\'es, S. Verma, and P. Viswanathan, Concerning the vector-valued fractal interpolation functions on the Sierpi\'nski gasket, Mediterr. J. Math., 18 (2021) article no. 202.
\bibitem{Mnew2} M. A. Navascues, S. Verma, Non-stationary alpha-fractal surfaces, Mediterr. J. Math., 20(1) (2023) 48.
\bibitem{KPV} K.K. Pandey, P. V. Viswanathan, Multivariate Fractal Functions in Some Complete Function Spaces and Fractional Integral of Continuous Fractal Functions,  Fractal and Fractional, 5(4), 185. 
\bibitem{PV11} S. A. Prasad, S. Verma, Fractal interpolation functions on products of the Sierpinski gaskets, Chaos, Solitons and Fractals 166 (2023) 112988.
\bibitem {Ruan} H. J. Ruan, Q. Xu, Fractal interpolation surfaces on rectangular grids, Bull. Aust. Math. Soc. 91 (2015) 435-446.
\bibitem{HWK} H. J. Ruan, W. Y. Su, K. Yaoc, Box dimension and fractional integral of linear fractal interpolation functions, Acta Math. Sinica (Chinese Ser.) 56 (2013), no. 5, 693–698.
\bibitem{SQ1} S. Schonefeld, Schauder bases in spaces of differentiable functions, Bull. Amer. Math. Soc. 75 (1969) 586-590.
\bibitem{Mverma1} Manuj V., A. Priyadarshi, S. Verma, Analytical and dimensional properties of fractal interpolation functions on the Sierpiński gasket, Fract. Calc. Appl. Anal., 26 (2023) 1294–1325.
\bibitem{Mverma2} Manuj V., A. Priyadarshi, Dimensions of new fractal functions and associated measures, Numer. Algorithms (2023) 1-30.




\bibitem{SQ2} S. Schonefeld, Schauder bases in the Banach spaces $C^k(T^q)$, Trans. Amer. Math. Soc. 165 (1972), 309–318.

\bibitem{MYUN} W. Metzler, C. H. Yun, Construction of fractal interpolation surfaces on rectangular grids, Internat. J. Bifur. Chaos 20 (2010) 4079–4086.

\bibitem{SV} S. Verma, P. Viswanathan, A fractalization of rational trigonometric functions, Mediterr. J. Math. 17 (3) (2020) 1-23.
\bibitem{SP} S. Verma, P. Viswanathan, A fractal operator associated with bivariate fractal interpolation functions on rectangular grids, Results in Mathematics 75(1) (2020) 28.
\bibitem{VV3} S. Verma, P. Viswanathan, Parameter identification for a class of bivariate fractal interpolation functions and constrained approximation, Numerical Functional Analysis and Optimization 41(9) (2020) 1109-1148.
\bibitem{Zhao} N. Zhao, Construction and application of fractal interpolation surfaces, The Visual Computer 12(3) (1996) 132-146.
\end{thebibliography}

\end{document}